\documentclass{amsart}

\usepackage{amsfonts}
\usepackage{ifthen}
\usepackage{amsthm}
\usepackage{amsmath}
\usepackage{graphicx}
\usepackage{amscd,amssymb,amsthm}
\usepackage{graphicx}

\newcounter{minutes}
\divide\time by 60
\newcounter{hours}
\multiply\time by 60 \addtocounter{minutes}{-\time}

\newtheorem{lemma}{Lemma}
\newtheorem{theorem}{Theorem}

\newcommand{\Real}{\operatorname{Re}}


\begin{document}
\title[Radii of convexity of Lommel and Struve functions]{Radii of convexity
of some Lommel and Struve functions}
\author[\'A. Baricz]{\'Arp\'ad Baricz}
\address{Department of Economics, Babe\c{s}-Bolyai University, Cluj-Napoca
400591, Romania}
\address{Institute of Applied Mathematics, \'Obuda University, 1034 Budapest, Hungary}
\email{bariczocsi@yahoo.com}

\author[N. Ya\u{g}mur]{Nihat Ya\u{g}mur}
\address{Department of Mathematics, Erzincan University, Erzincan 24000,
Turkey}
\email{nhtyagmur@gmail.com}

\begin{abstract}
The radii of convexity of some Lommel and Struve functions of the first kind are determined.
For both of Lommel and Struve functions three different
normalizations are applied in such a way that the resulting functions are
analytic in the unit disk of the complex plane. Some results on the zeros of the derivatives of some Lommel and Struve
functions of the first kind are also deduced, which may be of independent interest.
\end{abstract}

\maketitle

\def\thefootnote{}
\footnotetext{ \texttt{File:~\jobname .tex,
         printed: \number\year-\number\month-\number\day,
          \thehours.\ifnum\theminutes<10{0}\fi\theminutes}
} \makeatletter\def\thefootnote{\@arabic\c@footnote}\makeatother

%=======================================================================================================================================================

%=======================================================================================================================================================

\section{\bf Introduction}

Let $\mathbb{D}_{r}$ be the open disk $\left\{ {z\in \mathbb{C}:\left\vert
z\right\vert <r}\right\} ,$ where $r>0.$ As usual, with $\mathcal{A}$ we denote the class of analytic functions $f:\mathbb{D}%
_{r}\rightarrow \mathbb{C}$ which satisfy the usual normalization conditions
$f(0)=f^{\prime }(0)-1=0.$ Let us denote by $\mathcal{S}$ the class of functions
belonging to $\mathcal{A}$ which are univalent in $\mathbb{D}_{r}$ and let $%
\mathcal{K}(\alpha )$ be the subclass of $\mathcal{S}$ consisting of
functions which are convex of order $\alpha $ in $\mathbb{D}_{r},$ where $%
0\leq \alpha <1.$ The analytic characterization of this class of functions
is
\begin{equation*}
\mathcal{K}(\alpha )=\left\{ f\in \mathcal{S}\ :\ \Real \left( 1+\frac{%
zf^{\prime \prime }(z)}{f^{\prime }(z)}\right) >\alpha \ \ \mathrm{for\ all}%
\ \ z\in \mathbb{D}_{r}\right\} ,
\end{equation*}%
and we adopt the convention $\mathcal{K}=\mathcal{K}(0)$. The real number
\begin{equation*}
r_{\alpha }^{c}(f)=\sup \left\{ r>0\ :\ \Real \left( 1+\frac{zf^{\prime \prime
}(z)}{f^{\prime }(z)}\right) >\alpha \ \ \mathrm{\ for\ all}\ \ z\in \mathbb{%
D}_{r}\right\} ,
\end{equation*}%
is called the radius of convexity of order $\alpha $ of the function $f.$
It is worth to mention that $r^{c}(f)=r_{0}^{c}(f)$ is the largest radius such that the image
region $f(\mathbb{D}_{r^{c}(f)})$ is a convex domain in ${\mathbb{C}}$.

In this paper our aim is to consider two classical special functions, the Lommel function of the
first kind $s_{\mu ,\nu }$ and the Struve function of the first kind $%
\mathbf{H}_{\nu }$. They are explicitly defined in terms of the
hypergeometric function $\,_{1}F_{2}$ by
\begin{equation*}
s_{\mu ,\nu }(z)=\frac{z^{\mu +1}}{(\mu -\nu +1)(\mu +\nu +1)}%
\,_{1}F_{2}\left( 1;\frac{\mu -\nu +3}{2},\frac{\mu +\nu +3}{2};-\frac{z^{2}%
}{4}\right) ,\ \ \frac{1}{2}(-\mu \pm \nu -3)\not\in \mathbb{N},
\end{equation*}%
and
\begin{equation*}
\mathbf{H}_{\nu }(z)=\frac{\left( \frac{z}{2}\right) ^{\nu +1}}{\sqrt{\frac{%
\pi }{4}}\,\Gamma \left( \nu +\frac{3}{2}\right) }\,_{1}F_{2}\left( 1;\frac{3%
}{2},\nu +\frac{3}{2};-\frac{z^{2}}{4}\right) ,\ \ -\nu -\frac{3}{2}\not\in
\mathbb{N}.
\end{equation*}%
A common feature of these functions is that they are solutions of
inhomogeneous Bessel differential equations \cite{Wat}. Indeed, the Lommel
function of the first kind $s_{\mu ,\nu }$ is a solution of
\begin{equation*}
z^{2}w''(z)+zw'(z)+(z^{2}-{\nu }^{2})w(z)=z^{\mu +1}
\end{equation*}%
while the Struve function $\mathbf{H}_{\nu }$ obeys
\begin{equation*}
z^{2}w''(z)+zw'(z)+(z^{2}-{\nu }^{2})w(z)=\frac{%
4\left( \frac{z}{2}\right) ^{\nu +1}}{\sqrt{\pi }\Gamma \left( \nu +\frac{1}{%
2}\right) }.
\end{equation*}%
We refer to Watson's treatise \cite{Wat} for comprehensive information about
these functions and recall here briefly some contributions. In 1970 Steinig 
\cite{steinig} investigated the real zeros of the Struve function $\mathbf{H}_{\nu},$ while in 1972 he
\cite{stein} examined the sign of $s_{\mu ,\nu }(z)$ for real $\mu ,\nu $
and positive $z$. He showed, among other things, that for $\mu <\frac{1}{2}$
the function $s_{\mu ,\nu }$ has infinitely many changes of sign on $%
(0,\infty )$. In 2012 Koumandos and Lamprecht \cite{Kou} obtained sharp
estimates for the location of the zeros of $s_{\mu -\frac{1}{2},\frac{1}{2}}$
when $\mu \in (0,1)$. The Tur\'{a}n type inequalities for $s_{\mu -\frac{1}{2%
},\frac{1}{2}}$ were established in \cite{Bar2} while those for the Struve
function were proved in \cite{BPS}. Geometric properties of the Lommel function
$s_{\mu -\frac{1}{2},\frac{1}{2}}$ and of the Struve
function $\mathbf{H}_{\nu }$ were obtained in \cite{bdoy,Bar3} and in \cite{H-Ny,N-H}, respectively.
Motivated by those results, in this paper we are interested on the radii of convexity of certain
analytic functions related to the classical special functions under
discussion. Since neither $s_{\mu -\frac{1}{2},\frac{1}{2}}$, nor $\mathbf{H}%
_{\nu }$ belongs to $\mathcal{A}$, first we perform some natural
normalizations, as in \cite{bdoy}. We define three functions related to $s_{\mu -\frac{1}{%
2},\frac{1}{2}}$:
\begin{equation*}
f_{\mu }(z)=f_{\mu -\frac{1}{2},\frac{1}{2}}(z)=\left( \mu (\mu +1)s_{\mu -%
\frac{1}{2},\frac{1}{2}}(z)\right) ^{\frac{1}{\mu +\frac{1}{2}}},
\end{equation*}%
\begin{equation*}
g_{\mu }(z)=g_{\mu -\frac{1}{2},\frac{1}{2}}(z)=\mu (\mu +1)z^{-\mu +\frac{1%
}{2}}s_{\mu -\frac{1}{2},\frac{1}{2}}(z)
\end{equation*}%
and%
\begin{equation*}
h_{\mu }(z)=h_{\mu -\frac{1}{2},\frac{1}{2}}(z)=\mu (\mu +1)z^{\frac{3-2\mu
}{4}}s_{\mu -\frac{1}{2},\frac{1}{2}}(\sqrt{z}).
\end{equation*}%
Similarly, we associate with $\mathbf{H}_{\nu }$ the functions
\begin{equation*}
u_{\nu }(z)=\left( \sqrt{\pi }2^{\nu }\Gamma \left( \nu +\frac{3}{2}\right)
\mathbf{H}_{\nu }(z)\right) ^{\frac{1}{\nu +1}},
\end{equation*}%
\begin{equation*}
v_{\nu }(z)=\sqrt{\pi }2^{\nu }z^{-\nu }\Gamma \left( \nu +\frac{3}{2}%
\right) \mathbf{H}_{\nu }(z)
\end{equation*}%
and
\begin{equation*}
w_{\nu }(z)=\sqrt{\pi }2^{\nu }z^{\frac{1-\nu }{2}}\Gamma \left( \nu +\frac{3%
}{2}\right) \mathbf{H}_{\nu }(\sqrt{z}).
\end{equation*}

Clearly the functions $f_{\mu}$, $g_{\mu}$, $h_{\mu}$, $u_{\nu }$, $v_{\nu }$ and $w_{\nu }$ belong to the class $\mathcal{A}$. The
main results in the present paper concern the exact values of the radii of
convexity for these six functions, for some ranges of the parameters. The paper
is organized as follows: in this section we give some preliminaries which we will need in the sequel.
The next section contains the radii of convexity of the functions $f_{\mu},$ $g_{\mu}$ and $h_{\mu},$ while
section 3 contains the radii of convexity of the functions $u_{\nu }$, $v_{\nu }$ and $w_{\nu }.$ As we can see in the
proof of the first parts of the main theorems the interlacing property of the zeros of the Lommel and Struve functions and their derivatives play an important role. These interlacing properties actually yield that under some conditions on the parameters the zeros of the derivatives of Lommel and Struve functions are all real. Section 4 contains a discussion about these things and also some alternative proofs based on a famous result of P\'olya. This paper is a direct continuation of the recent works \cite{bdoy,bdm,BAS}.

The following preliminary result is the content of Lemmas 1 and 2 in \cite{Bar2} and is one of the key tools
in the proof of our main results.

\begin{lemma}\label{lem1}
Let
\begin{equation*}
\varphi _{k}(z)=\,_{1}F_{2}\left( 1;\frac{\mu -k+2}{2},\frac{\mu -k+3}{2};-%
\frac{z^{2}}{4}\right)
\end{equation*}%
where ${z\in \mathbb{C}}$, ${\mu \in \mathbb{R}}$ and ${k\in }\left\{
0,1,\dots \right\} $ such that ${\mu -k}$ is not in $\left\{ 0,-1,\dots
\right\} $. Then, $\varphi _{k}$ is an entire function of order $\rho =1.$ Consequently, the Hadamard's
factorization of $\varphi _{k}$ is of the form%
$$
\varphi _{k}(z)=\prod\limits_{n\geq 1}\left( 1-\frac{z^{2}}{z_{\mu
,k,n}^{2}}\right) ,  \label{1.6}
$$
where $z_{\mu ,k,n}$ is the $n$th positive zero of the
function $\varphi _{k}$ and the infinite product is absolutely convergent.
Moreover, for $z,$ ${\mu }$ and $k$ as above, we have
\begin{equation*}
(\mu -k+1)\varphi _{k+1}(z)=(\mu -k+1)\varphi _{k}(z)+z\varphi _{k}^{\prime
}(z),
\end{equation*}%
\begin{equation*}
\sqrt{z}s_{\mu -k-\frac{1}{2},\frac{1}{2}}(z)=\frac{z^{\mu -k+1}}{(\mu
-k)(\mu -k+1)}\varphi _{k}(z).
\end{equation*}
\end{lemma}

We also recall that by definition the real entire function $\phi,$
defined by
$$\phi(z)=\phi(z;t)=\sum_{n\geq0}a_n(t)\frac{z^n}{n!},$$ is said
is in the Laguerre-P\'olya class if
$\phi(z)$ can be expressed in the form
$$\phi(z)=cz^de^{-\alpha z^2+\beta z}\prod_{n=1}^{\omega}\left(1-\frac{z}{z_n}\right)e^{\frac{z}{z_n}},\ \ \ 0\leq\omega\leq\infty,$$
where $c$ and $\beta$ are real, $z_n$'s are real and nonzero for all
$n\in\{1,2,{\dots},\omega\},$ $\alpha\geq0,$ $d$ is a nonnegative
integer and $\sum_{n=1}^{\omega}z_i^{-2}<\infty.$ If $\omega=0,$
then, by convention, the product is defined to be $1.$

\section{\bf Radii of convexity of some Lommel functions of the first kind}
\setcounter{equation}{0}

The first main result we establish reads as follows.

\begin{theorem}
\label{theo1} Let $\mu\in(-1,1),$ $\mu\neq0$ and suppose that $0\leq\alpha<1.$ Moreover, let $\xi _{\mu ,1}$ and $\xi _{\mu ,1}^{\prime }$ denote the first
positive zeros of $s_{\mu -\frac{1}{2},\frac{1}{2}}$ and $s_{\mu -\frac{1}{2}%
,\frac{1}{2}}^{\prime }$, respectively. Then the following statements hold:
\begin{enumerate}
\item[\textbf{a)}] If $\mu \neq -\frac{1}{2},$ then
the radius of convexity of order $\alpha $ of the function $f_{\mu }$ is the
smallest positive root of the equation
\begin{equation*}
1+\frac{r\,s_{\mu -\frac{1}{2},\frac{1}{2}}^{\prime \prime }(r)}{s_{\mu -%
\frac{1}{2},\frac{1}{2}}^{\prime }(r)}+\left( \frac{1}{\mu +\frac{1}{2}}%
-1\right) \frac{r\,s_{\mu -\frac{1}{2},\frac{1}{2}}^{\prime }(r)}{s_{\mu -%
\frac{1}{2},\frac{1}{2}}(r)}=\alpha.
\end{equation*}

\item[\textbf{b)}] The radius of convexity of
order $\alpha $ of the function $g_{\mu }$ is the smallest positive root of
the equation
\begin{equation*}
-\mu +\frac{1}{2}+\frac{r\,s_{\mu -\frac{1}{2},\frac{1}{2}}^{\prime \prime
}(r)}{s_{\mu -\frac{1}{2},\frac{1}{2}}^{\prime }(r)}=\alpha.
\end{equation*}

\item[\textbf{c)}] The radius of convexity of
order $\alpha $ of the function $h_{\mu }$ is the smallest positive root of
the equation
\begin{equation*}
-\mu +\frac{1}{2}+\frac{r^{\frac{1}{2}}\,s_{\mu -\frac{1}{2},\frac{1}{2}%
}^{\prime \prime }(r^{\frac{1}{2}})}{s_{\mu -\frac{1}{2},\frac{1}{2}%
}^{\prime }(r^{\frac{1}{2}})}=2\alpha.
\end{equation*}
\end{enumerate}

Moreover, we have the inequalities $r_{\alpha }^{c}(f_{\mu })<\xi _{\mu ,1}',$ $r_{\alpha }^{c}(g_{\mu })<\xi _{\mu ,1}',$ $r_{\alpha }^{c}(h_{\mu })<\xi _{\mu ,1}^{\prime }$ and $\xi _{\mu ,1}^{\prime }<\xi _{\mu ,1}.$
\end{theorem}

\begin{proof}[\bf Proof]

{\bf a)} By using Lemma \ref{lem1} we know that the Hadamard factorization of $%
s_{\mu -\frac{1}{2},\frac{1}{2}}(z)$ is
\begin{equation}
s_{\mu -\frac{1}{2},\frac{1}{2}}(z)=\frac{z^{\mu +\frac{1}{2}}}{\mu (\mu +1)}%
\varphi _{0}(z)=\frac{z^{\mu +\frac{1}{2}}}{\mu (\mu +1)}\prod%
\limits_{n\geq1}\left( 1-\frac{z^{2}}{\xi _{\mu ,n}^{2}}\right) ,
\label{s1}
\end{equation}%
where $\xi _{\mu ,n}$ denotes the $n$th positive zero of the function $\varphi _{0}.$
Moreover, it is known that if $\mu \in (0,1),$ then the zeros of the function $\varphi _{0}$
are real (see \cite[Lemma 2.1]{Kou} and \cite{Bar2}). We also note that the
infinite sum representation of the Lommel function $s_{\mu -\frac{1}{2},%
\frac{1}{2}}$ is
\begin{equation*}
s_{\mu -\frac{1}{2},\frac{1}{2}}(z)=\frac{z^{\mu +\frac{1}{2}}}{\mu (\mu +1)}%
\sum_{n\geq0}\frac{\left( -1\right) ^{n}}{\left( \frac{\mu +2}{2}%
\right) _{n}\left( \frac{\mu +3}{2}\right) _{n}}\left( \frac{z}{2}\right)
^{2n},  \label{s2}
\end{equation*}%
where $(\lambda )_{n}=\Gamma (\lambda +n)/\Gamma (\lambda )=\lambda (\lambda
+1)\dots(\lambda +n-1)$ is the well-known Pochhammer (or Appell) symbol. Thus, we obtain%
\begin{equation*}
s_{\mu -\frac{1}{2},\frac{1}{2}}^{\prime }(z)=\frac{\left( \mu +\frac{1}{2}%
\right) z^{\mu -\frac{1}{2}}}{\mu (\mu +1)}\sum_{n\geq0}\frac{\left(
-1\right) ^{n}\left( \mu +\frac{1}{2}+2n\right) }{\left( \mu +\frac{1}{2}%
\right) \left( \frac{\mu +2}{2}\right) _{n}\left( \frac{\mu +3}{2}\right)
_{n}}\left( \frac{z}{2}\right) ^{2n}
\end{equation*}%
and%
\begin{equation*}
\frac{\mu (\mu +1)}{\mu +\frac{1}{2}}z^{-\mu +\frac{1}{2}%
}s_{\mu -\frac{1}{2},\frac{1}{2}}^{\prime }(z)=1+\sum_{n\geq1}\frac{%
\left( -1\right) ^{n}\left( \mu +\frac{1}{2}+2n\right) z^{2n}}{\left( \mu +%
\frac{1}{2}\right) 2^{2n}\left( \frac{\mu +2}{2}\right) _{n}\left( \frac{\mu
+3}{2}\right) _{n}}.
\end{equation*}%
Taking into consideration the well-known limit
\begin{equation}
\lim_{n\longrightarrow \infty }\frac{\log \Gamma \left( n+c\right) }{n\log n}%
=1,  \label{L1}
\end{equation}%
where $c$ is a positive constant, and \cite[p. 6, Theorems 2]{LV} we
infer that the above entire function is of growth order $\rho =\frac{1%
}{2}$. Namely, for $\mu \in (0,1)$ we have that as $n\longrightarrow \infty $%
\begin{equation*}
\frac{n\log n}{\log \left( \mu +\frac{1}{2}\right) +2n\log 2+\log \left(
\frac{\mu +2}{2}\right) _{n}+\log \left( \frac{\mu +3}{2}\right) _{n}-\log
(\mu +\frac{1}{2}+2n)}\longrightarrow \frac{1}{2}.
\end{equation*}
Now, by applying Hadamard's Theorem \cite[p. 26]{LV} we can write the
infinite product representation of $s_{\mu -\frac{1}{2},\frac{1}{2}}^{\prime
}$ as follows
\begin{equation}
s_{\mu -\frac{1}{2},\frac{1}{2}}^{\prime }(z)=\frac{\left( \mu +\frac{1}{2}%
\right) z^{\mu -\frac{1}{2}}}{\mu (\mu +1)}\prod\limits_{n\geq1}\left( 1-\frac{z^{2}}{\xi _{\mu ,n}^{\prime 2}}\right),  \label{s5}
\end{equation}%
where $\xi _{\mu ,n}^{\prime }$ denotes the $n$th positive zero of the function $s_{\mu -%
\frac{1}{2},\frac{1}{2}}'.$

Note that (see \cite{Bar2}) the function $z\mapsto\varphi _{0}(z)=\mu (\mu +1)z^{-\mu -\frac{1}{2}}s_{\mu -%
\frac{1}{2},\frac{1}{2}}(z)$ belongs to the Laguerre-P\'{o}lya class of entire functions (since the exponential factors in the infinite product are canceled because of the symmetry of the zeros $\pm \xi_{\mu,n},$ $n\in\mathbb{N},$ with respect to the origin), and consequently it satisfies the
Laguerre inequality 
\begin{equation}
\left( \varphi _{0}^{(n)}(z)\right) ^{2}-\left( \varphi
_{0}^{(n-1)}(z)\right) \left( \varphi_{0}^{(n+1)}(z)\right) >0,  \label{s6}
\end{equation}%
where $\mu \in (0,1)$ and $z\in \mathbb{R}$. On the other hand, we have that%
\begin{equation}\label{phider}
\varphi _{0}^{\prime }(z)=\mu (\mu +1)z^{-\mu -\frac{3}{2}}\left[ zs_{\mu -%
\frac{1}{2},\frac{1}{2}}^{\prime }(z)-\left( \mu +\frac{1}{2}\right) s_{\mu -%
\frac{1}{2},\frac{1}{2}}(z)\right],
\end{equation}
\begin{equation*}
\varphi _{0}^{\prime \prime }(z)=\mu (\mu +1)z^{-\mu -\frac{5}{2}}\left[
z^{2}s_{\mu -\frac{1}{2},\frac{1}{2}}^{\prime \prime }(z)-2\left( \mu +\frac{%
1}{2}\right) zs_{\mu -\frac{1}{2},\frac{1}{2}}^{\prime }(z)+\left( \mu +%
\frac{1}{2}\right) \left( \mu +\frac{3}{2}\right) s_{\mu -\frac{1}{2},\frac{1%
}{2}}(z)\right] ,
\end{equation*}%
and thus the Laguerre inequality (\ref{s6}) for $n=1$ is equivalent to%
\begin{equation*}
\mu (\mu +1)z^{-2\mu -3}\left[ z^{2}\left( s_{\mu -\frac{1}{2},\frac{1}{2}%
}^{\prime }(z)\right) ^{2}-z^{2}s_{\mu -\frac{1}{2},\frac{1}{2}}(z)s_{\mu -%
\frac{1}{2},\frac{1}{2}}^{\prime \prime }(z)-\left( \mu +\frac{1}{2}\right)
\left( s_{\mu -\frac{1}{2},\frac{1}{2}}(z)\right) ^{2}\right] >0.
\end{equation*}%
This implies that%
$$
\left( s_{\mu -\frac{1}{2},\frac{1}{2}}^{\prime }(z)\right) ^{2}-s_{\mu -%
\frac{1}{2},\frac{1}{2}}(z)s_{\mu -\frac{1}{2},\frac{1}{2}}^{\prime \prime
}(z)>\frac{\left( \mu +\frac{1}{2}\right) \left( s_{\mu -\frac{1}{2},\frac{1%
}{2}}(z)\right) ^{2}}{z^{2}}>0
$$
for $\mu\in(0,1)$ and $z\in \mathbb{R}$, and thus $z\longmapsto
s_{\mu -\frac{1}{2},\frac{1}{2}}^{\prime }(z)/s_{\mu -\frac{1}{2},\frac{1}{2}%
}(z)$ is decreasing on $(0,\infty )\backslash \left\{ \xi _{\mu ,n}:n\in
\mathbb{N}\right\} .$ Since the zeros $\xi _{\mu ,n}$ of the Lommel function
$s_{\mu -\frac{1}{2},\frac{1}{2}}$ are real and simple, $s_{\mu -\frac{1}{2},%
\frac{1}{2}}^{\prime }(z)$ does not vanish in $\xi _{\mu ,n},$ $n\in \mathbb{%
N}$. Thus, for a fixed $k\in \mathbb{N}$ the function $z\longmapsto s_{\mu -%
\frac{1}{2},\frac{1}{2}}^{\prime }(z)/s_{\mu -\frac{1}{2},\frac{1}{2}}(z)$
takes the limit $\infty $ when $z\searrow \xi _{\mu ,k-1},$ and the limit $%
-\infty $ when $z\nearrow \xi _{\mu ,k}.$ Moreover, since $z\longmapsto
s_{\mu -\frac{1}{2},\frac{1}{2}}^{\prime }(z)/s_{\mu -\frac{1}{2},\frac{1}{2}%
}(z)$ is decreasing on $(0,\infty )\backslash \left\{ \xi _{\mu ,n}:n\in
\mathbb{N}\right\} $ it results that in each interval $(\xi _{\mu ,k-1},\xi
_{\mu ,k})$ its restriction intersects the horizontal line only once, and
the abscissa of this intersection point is exactly $\xi _{\mu,k}^{\prime }$.
Consequently, the zeros $\xi _{\mu ,n}$ and $\xi _{\mu ,n}^{\prime }$ interlace. Here
we used the convention that $\xi _{\mu ,0}=0$.

On the other hand, it is known that \cite{BAS} if ${z\in \mathbb{C}}$ and $%
\beta $ ${\in \mathbb{R}}$ are such that $\beta >{\left\vert z\right\vert }$%
, then%
\begin{equation}
\frac{{\left\vert z\right\vert }}{\beta -{\left\vert z\right\vert }}\geq \Real
\left( \frac{z}{\beta -z}\right) .  \label{s7}
\end{equation}
Observe also that
\begin{equation*}
1+\frac{zf_{\mu }^{\prime \prime }(z)}{f_{\mu }^{\prime }(z)}=1+\frac{%
zs_{\mu -\frac{1}{2},\frac{1}{2}}^{\prime \prime }(z)}{s_{\mu -\frac{1}{2},%
\frac{1}{2}}^{\prime }(z)}+\left( \frac{1}{\mu +\frac{1}{2}}-1\right) \frac{%
zs_{\mu -\frac{1}{2},\frac{1}{2}}^{\prime }(z)}{s_{\mu -\frac{1}{2},\frac{1%
}{2}}(z)}\text{\ .\ }
\end{equation*}%
By means of (\ref{s1}) and (\ref{s5}) we have%
\begin{equation}
\frac{zs_{\mu -\frac{1}{2},\frac{1}{2}}^{\prime }(z)}{s_{\mu -\frac{1}{2},%
\frac{1}{2}}(z)}=\mu +\frac{1}{2}-\sum_{n\geq 1}\frac{2z^{2}}{\xi
_{\mu ,n}^{2}-z^{2}},\text{ \ \ \ }1+\frac{zs_{\mu -\frac{1}{2},\frac{1}{2}%
}^{\prime \prime }(z)}{s_{\mu -\frac{1}{2},\frac{1}{2}}^{\prime }(z)}=\mu +%
\frac{1}{2}-\sum_{n\geq 1}\frac{2z^{2}}{\xi _{\mu ,n}^{\prime 2}-z^{2}}%
,  \label{2.1}
\end{equation}%
and it follows that%
\begin{equation*}
1+\frac{zf_{\mu }^{\prime \prime }(z)}{f_{\mu }^{\prime }(z)}=1-\left(
\frac{1}{\mu +\frac{1}{2}}-1\right) \sum_{n\geq1}\frac{2z^{2}}{\xi
_{\mu ,n}^{2}-z^{2}}-\sum_{n\geq1}\frac{2z^{2}}{\xi _{\mu ,n}^{\prime
2}-z^{2}}.
\end{equation*}%
Now, suppose that $\mu \in (0,\frac{1}{2}].$ By using (\ref{s7}), we obtain for all $z\in \mathbb{D}_{\xi _{\mu ,1}^{\prime }}$ the
inequality%
\begin{equation*}
\Real\left( 1+\frac{zf_{\mu}''(z)}{f_{\mu}'(z)}\right) \geqslant 1-\left( \frac{1}{\mu +\frac{1}{2}}-1\right)
\sum_{n\geq1}\frac{2r^{2}}{\xi _{\mu ,n}^{2}-r^{2}}%
-\sum_{n\geq1}\frac{2r^{2}}{\xi _{\mu ,n}^{\prime 2}-r^{2}},
\label{s8}
\end{equation*}%
where ${\left\vert z\right\vert =r.}$ Moreover, observe that if we use the
inequality \cite[Lemma 2.1]{BAS}%
\begin{equation}
\lambda \Real\left( \frac{z}{a-z}\right) -\Real\left( \frac{z}{b-z}%
\right) \geqslant \lambda \frac{{\left\vert z\right\vert }}{a-{\left\vert
z\right\vert }}-\frac{{\left\vert z\right\vert }}{b-{\left\vert z\right\vert
}},  \label{2.2}
\end{equation}%
where $a>b>0,$ $\lambda \in \left[ 0,1\right] $ and $z\in \mathbb{C}$ such
that ${\left\vert z\right\vert <b}$, then we get that the above inequality
is also valid when $\mu \in\left(\frac{1}{2},1\right)$. Here we used that the zeros $\xi
_{\mu ,n}$ and $\xi _{\mu ,n}^{\prime }$ interlace. Thus, for ${r}\in \left( 0,\xi _{\mu ,1}^{\prime }\right) $ we have
\begin{equation*}
\inf_{z\in \mathbb{D}_{r}}\left\{ \Real\left( 1+\frac{zf_{\mu}''(z)}{f_{\mu}'(z)}\right) \right\} =1+\frac{%
rf_{\mu}''(r)}{f_{\mu}'(r)}.
\end{equation*}%
On the other hand, the function $F_{\mu }:\left( 0,\xi _{\mu ,1}^{\prime
}\right) \longrightarrow \mathbb{R}$, defined by%
\begin{equation*}
F_{\mu }(r)=1+\frac{rf_{\mu}''(r)}{f_{\mu}'(r)},
\end{equation*}%
is strictly decreasing for all $\mu\in(0,1).$ Namely, we have
\begin{eqnarray*}
F_{\mu }^{\prime }(r) &=&-\left( \frac{1}{\mu +\frac{1}{2}}-1\right)
\sum_{n\geq1}\frac{4r\xi _{\mu ,n}^{2}}{\left( \xi _{\mu
,n}^{2}-r^{2}\right) ^{2}}-\sum_{n\geq1}\frac{4r\xi _{\mu ,n}^{\prime
2}}{\left( \xi _{\mu ,n}^{\prime 2}-r^{2}\right) ^{2}} \\
&<&\sum_{n\geq1}\frac{4r\xi _{\mu ,n}^{2}}{\left( \xi _{\mu
,n}^{2}-r^{2}\right) ^{2}}-\sum_{n\geq1}\frac{4r\xi _{\mu ,n}^{\prime
2}}{\left( \xi _{\mu ,n}^{\prime 2}-r^{2}\right) ^{2}}<0
\end{eqnarray*}%
for $\mu\in\left(\frac{1}{2},1\right)$ and ${r}\in \left( 0,\xi _{\mu ,1}^{\prime }\right).$ Here we
used again that the zeros $\xi _{\mu ,n}$ and $\xi _{\mu ,n}^{\prime }$
interlace, and for all $n\in \mathbb{N}$, $\mu\in(0,1)$ and $r<\sqrt{\xi _{\mu ,n}\xi
_{\mu ,n}^{\prime }}$ we have that
\begin{equation*}
\xi _{\mu ,n}^{2}\left( \xi _{\mu,n}^{\prime 2}-r^{2}\right) ^{2}<\xi _{\mu
,n}^{\prime 2}\left( \xi _{\mu,n}^{2}-r^{2}\right) ^{2}.
\end{equation*}
Observe that when $\mu\in\left(0,\frac{1}{2}\right]$ and $r>0$ we have also that $F_{\mu}'(r)<0,$ and thus $F_{\mu}$ is indeed decreasing for all $\mu\in(0,1).$ Now, since $\lim_{r\searrow 0}F_{\mu }(r)=1>\alpha $ and $\lim_{r\nearrow \xi
_{\mu ,1}^{\prime }}F_{\mu }(r)=-\infty ,$ in view of the minimum principle for harmonic functions it follows that for $\mu\in(0,1)$ and $z\in \mathbb{%
D}_{r_{1}}$ we have%
\begin{equation}
\Real\left( 1+\frac{zf_{\mu}''(z)}{f_{\mu}'(z)}\right) >\alpha
\label{f1}
\end{equation}%
if and only if $r_{1}$ is the unique root of
$$1+\frac{rf_{\mu}''(r)}{f_{\mu}'(r)}=\alpha$$
situated in $\left( 0,\xi _{\mu ,1}'\right).$ This completes the
proof of part \textbf{a)} of our theorem when $\mu \in (0,1)$.

Now we prove that \eqref{f1} also holds when $\mu \in \left(
-1,0\right) .$

In order to do this, suppose that $\mu \in (0,1)$ and repeat the above proof,
substituting $\mu $ by $\mu -1$, $\varphi _{0}$ by the function $\varphi _{1}
$ and taking into account that the $n$th positive zero of $\varphi _{1}$ and
$\varphi _{1}^{\prime },$ denoted by $\zeta _{\mu ,n}$ and $\zeta _{\mu
,n}^{\prime },$ interlace, since $\varphi_1$ belongs also to the Laguerre-P\'olya
class of entire functions (see \cite{Bar2}). It is worth mentioning that%
\begin{equation*}
\Real \left( 1+\frac{zf_{\mu -1}^{\prime \prime }(z)}{f_{\mu -1}^{\prime }(z)%
}\right) \geq 1-\left( \frac{1}{\mu -\frac{1}{2}}-1\right)
\sum_{n\geq 1}\frac{2r^{2}}{\zeta _{\mu ,n}^{2}-r^{2}}%
-\sum_{n\geq1}\frac{2r^{2}}{\zeta _{\mu ,n}^{\prime 2}-r^{2}}
\end{equation*}%
holds for $\mu \in \left(0,1\right) $ with $\mu \neq \frac{1}{2}$.
In this case we use again the minimum principle for harmonic functions to ensure
that \eqref{f1} is valid for $\mu -1$ instead of $\mu .$ Thus, replacing $\mu $ by $\mu +1$, we obtain the statement of the part \textbf{a)%
} for $\mu \in \left( -1,0\right) $ with $\mu \neq -\frac{1}{2}.$

\textbf{b) }Observe that
$$
1+\frac{zg_{\mu }^{\prime \prime }(z)}{g_{\mu }^{\prime }(z)}=-\mu +\frac{1%
}{2}+\frac{zs_{\mu -\frac{1}{2},\frac{1}{2}}^{\prime \prime }(z)}{s_{\mu -%
\frac{1}{2},\frac{1}{2}}^{\prime }(z)}\text{\ .\ }
$$
By means of (\ref{2.1}) we have%
$$
1+\frac{zg_{\mu }^{\prime \prime }(z)}{g_{\mu }^{\prime }(z)}%
=-\sum_{n\geq1}\frac{2z^{2}}{\xi _{\mu ,n}^{\prime 2}-z^{2}}.
$$
Now, suppose that $\mu \in \left( 0,1\right) .$ By using the inequality (\ref%
{s7}), for all $z\in \mathbb{D}_{\xi _{\mu ,1}^{\prime }}$ we obtain the
inequality%
$$
\Real\left( 1+\frac{zg_{\mu}''(z)}{g_{\mu}'(z)}\right) \geqslant -\sum_{n\geq 1}\frac{2r^{2}}{\xi _{\mu
,n}^{\prime 2}-r^{2}},
$$
where ${\left\vert z\right\vert =r.}$ Thus, for ${r}\in \left( 0,\xi _{\mu
,1}^{\prime }\right) $ we get%
\begin{equation*}
\inf_{z\in \mathbb{D}_{r}}\left\{ \Real\left( 1+\frac{zg_{\mu}''(z)}{g_{\mu}'(z)}\right) \right\} =1+\frac{rg_{\mu}''(r)}{g_{\mu}'(r)}.
\end{equation*}%
The function $G_{\mu }:\left( 0,\xi _{\mu ,1}'\right)
\longrightarrow \mathbb{R}$, defined by%
\begin{equation*}
G_{\mu }(r)=1+\frac{rg_{\mu }^{\prime \prime }(r)}{g_{\mu }^{\prime }(r)},
\end{equation*}%
is strictly decreasing and $\lim_{r\searrow 0}G_{\mu }(r)=1>\alpha$, $%
\lim_{r\nearrow \xi _{\mu ,1}^{\prime }}G_{\mu }(r)=-\infty $. Consequently, in view of the minimum principle for harmonic functions
for $z\in \mathbb{D}_{r_{2}}$ we have that
\begin{equation*}
\Real\left( 1+\frac{zg_{\mu}''(z)}{g_{\mu}'(z)}\right) >\alpha
\end{equation*}%
if and only if $r_{2}$ is the unique root of
\begin{equation*}
1+\frac{rg_{\mu}''(r)}{g_{\mu}'(r)}=\alpha
\end{equation*}%
situated in $\left( 0,\xi _{\mu ,1}'\right) .$ For $\mu \in (-1,0)$
the proof goes along the lines introduced at the end of the proof of part \textbf{a)}, and thus we omit the details.

\textbf{c)} Observe that
$$
1+\frac{zh_{\mu }^{\prime \prime }(z)}{h_{\mu }^{\prime }(z)}=\frac{1-2\mu
}{4}+\frac{\sqrt{z}\,s_{\mu -\frac{1}{2},\frac{1}{2}}^{\prime \prime }(\sqrt{%
z})}{2s_{\mu -\frac{1}{2},\frac{1}{2}}^{\prime }(\sqrt{z})}.
$$
By means of (\ref{2.1}) we have%
$$
1+\frac{z\,h_{\mu }^{\prime \prime }(z)}{h_{\mu }^{\prime }(z)}%
=-\sum_{n\geq1}\frac{2z}{\xi _{\mu ,n}^{\prime 2}-z}.$$
Now, suppose that $\mu \in \left( 0,1\right) .$ By using (\ref%
{s7}), for all $z\in \mathbb{D}_{\xi _{\mu ,1}^{\prime }}$ we obtain the
inequality%
\begin{equation*}
\Real\left( 1+\frac{z\,h_{\mu }''(z)}{h_{\mu}'(z)}\right) \geqslant -\sum_{n\geq1}\frac{2r}{\xi _{\mu ,n}^{\prime
2}-r},  \label{h3}
\end{equation*}%
where ${\left\vert z\right\vert =r.}$ Thus, for ${r}\in \left( 0,\xi _{\mu
,1}^{\prime }\right) $ we get%
\begin{equation*}
\inf_{z\in \mathbb{D}_{r}}\left\{ \Real\left( 1+\frac{zh_{\mu}''(z)}{h_{\mu}'(z)}\right) \right\} =1+\frac{r\,h_{\mu
}^{\prime \prime }(r)}{h_{\mu }^{\prime }(r)}.
\end{equation*}%
The function $H_{\mu }:\left( 0,\xi _{\mu,1}'\right)
\longrightarrow \mathbb{R}$, defined by%
\begin{equation*}
H_{\mu }(r)=1+\frac{r\,h_{\mu}''(r)}{h_{\mu}'(r)},
\end{equation*}%
is strictly decreasing and $\lim_{r\searrow 0}H_{\mu }(r)=1>\alpha$, $%
\lim_{r\nearrow \xi _{\mu ,1}^{\prime }}H_{\mu }(r)=-\infty $. Consequently, in view of the minimum principle for harmonic functions
for $z\in \mathbb{D}_{r_{3}}$ we have that
\begin{equation*}
\Real\left( 1+\frac{z\,h_{\mu}''(z)}{h_{\mu}'(z)}\right) >\alpha
\end{equation*}%
if and only if $r_{3}$ is the unique root of
\begin{equation*}
1+\frac{r\,h_{\mu}''(r)}{h_{\mu}'(r)}=\alpha
\end{equation*}%
situated in $\left( 0,\xi _{\mu,1}'\right) .$ For $\mu \in (-1,0)$
the proof is similar, and we omit the details.
\end{proof}

\section{\bf Radii of convexity of Struve functions of the first kind}
\setcounter{equation}{0}

In this section our aim is to present the corresponding result about the radii of convexity of the functions,
related to Struve's one.

\begin{theorem}\label{theo2}
Let $|\nu|\leq\frac{1}{2}$ and $0\leq \alpha <1.$ Moreover, let $h_{\nu,1}$
and $h_{\nu,1}^{\prime }$ denote the first positive zeros of $\mathbf{H}_{\nu}$
and $\mathbf{H}_{\nu}^{\prime }$, respectively. Then, the following assertions are true:

\begin{enumerate}
\item[\textbf{a)}] The radius of convexity of
order $\alpha $ of the function $u_{\nu}$ is the smallest positive root of the
equation
\begin{equation*}
1+\frac{r\,\mathbf{H}_{\nu}^{\prime \prime }(r)}{\mathbf{H}_{\nu}^{\prime }(r)}%
+\left( \frac{1}{\nu+1}-1\right) \frac{r\,\mathbf{H}_{\nu}^{\prime }(r)}{\mathbf{%
H}_{\nu}(r)}=\alpha .
\end{equation*}%

\item[\textbf{b)}] The radius of convexity of
order $\alpha $ of the function $v_{\nu}$ is the smallest positive root of the
equation
\begin{equation*}
-\nu+\frac{r\,\mathbf{H}_{\nu}^{\prime \prime }(r)}{\mathbf{H}_{\nu}^{\prime }(r)}%
=\alpha .
\end{equation*}

\item[\textbf{c)}] The radius of convexity of
order $\alpha $ of the function $w_{\nu}$ is the smallest positive root of the
equation
\begin{equation*}
-\nu+\frac{r^{\frac{1}{2}}\,\mathbf{H}_{\nu}^{\prime \prime }(r^{\frac{1}{2}})}{%
\mathbf{H}_{\nu}^{\prime}(r^{\frac{1}{2}})}=2\alpha .
\end{equation*}%
\end{enumerate}
Moreover, we have the inequalities $r_{\alpha }^{c}(u_{\nu})<h_{\nu,1}^{\prime},$ $r_{\alpha }^{c}(v_{\nu})<h_{\nu,1}^{\prime },$ $r_{\alpha }^{c}(w_{\nu})<h_{\nu,1}^{\prime }$ and $h_{\nu,1}^{\prime }<h_{\nu,1}.$
\end{theorem}

\begin{proof}[\bf Proof]
\textbf{a)} As in the proof of Theorem \ref{theo1} we need to the Hadamard
factorizations of $\mathbf{H}_{\nu }$ and $\mathbf{H}_{\nu }^{\prime }.$ If $\left\vert \nu \right\vert \leq \frac{1}{2},$ then (see \cite[Lemma 1]%
{BPS}) the Hadamard factorization of the transcendental entire function $%
\mathcal{H}_{\nu }$, defined by
\begin{equation*}
\mathcal{H}_{\nu }(z)=\sqrt{\pi }2^{\nu }z^{-\nu -1}\Gamma \left( \nu +\frac{%
3}{2}\right) \mathbf{H}_{\nu }(z),
\end{equation*}%
reads as follows%
\begin{equation*}
\mathcal{H}_{\nu }(z)=\prod\limits_{n\geq1}\left( 1-\frac{z^{2}}{%
h_{\nu ,n}^{2}}\right) ,
\end{equation*}%
which implies that
\begin{equation}\label{H1}
\mathbf{H}_{\nu }(z)=\frac{z^{\nu +1}}{\sqrt{\pi }2^{\nu }\Gamma \left( \nu +%
\frac{3}{2}\right) }\prod\limits_{n\geq1}\left( 1-\frac{z^{2}}{h_{\nu
,n}^{2}}\right),
\end{equation}
where $h_{\nu ,n}$ stands for the $n$th positive zero of the Struve function
$\mathbf{H}_{\nu }.$ The infinite sum representation of the function $\mathbf{H}_{\nu}'$ is
$$
\mathbf{H}_{\nu}'(z)=\frac{(\nu +1)z^{\nu }}{\sqrt{\pi }2^{\nu
}\Gamma \left( \nu +\frac{3}{2}\right) }\sum_{n\geq0}\frac{\left(
-1\right) ^{n}(2n+\nu+1)}{(\nu +1)2^{2n}\left(\frac{3}{2}\right) _{n}\left( \nu
+\frac{3}{2}\right) _{n}}z^{2n},$$ which can be rewritten as
\begin{equation*}
\frac{\sqrt{\pi }2^{\nu }\Gamma \left( \nu +\frac{3}{2}\right)
}{(\nu +1)}z^{-\nu }\mathbf{H}_{\nu }^{\prime }(z)=1+\sum_{n\geq 1}
\frac{\left( -1\right) ^{n}(2n+\nu+1)}{(\nu +1)2^{2n}\left(\frac{3}{2}\right)
_{n}\left( \nu +\frac{3}{2}\right) _{n}}z^{2n}.
\end{equation*}
Taking into consideration the limit in (\ref{L1}) we
deduce that the above entire function is of growth order $\rho =\frac{1%
}{2}$. Namely, for $\left\vert \nu \right\vert \leq \frac{1}{2}$ we have
that as $n\longrightarrow \infty $%
\begin{equation*}
\frac{n\log n}{\log \left(\nu+1\right) +2n\log 2+\log \left(\frac{3}{2}\right)
_{n}+\log \left( \nu+\frac{3}{2}\right)_{n}-\log(2n+\nu+1)}\longrightarrow \frac{1}{2}.
\end{equation*}
Now, by applying Hadamard's Theorem \cite[p. 26]{LV} we can write the
infinite product representation of $\mathbf{H}_{\nu }^{\prime }$ as follows
\begin{equation}
\mathbf{H}_{\nu }^{\prime }(z)=\frac{(\nu +1)z^{\nu}}{\sqrt{\pi }2^{\nu
}\Gamma \left( \nu +\frac{3}{2}\right) }\prod\limits_{n\geq1}\left( 1-%
\frac{z^{2}}{h_{\nu,n}^{\prime 2}}\right)  \label{H3}
\end{equation}%
where $h_{\nu ,n}^{\prime }$ denotes the $n$th positive zero of the function $\mathbf{H}%
_{\nu }^{\prime }.$

It is worth to mention that the zeros of $\mathbf{H}_{\nu}$ are all real and simple when $|\nu|\leq \frac{1}{2},$ see \cite{steinig}. Thus, the function $x\longmapsto \mathcal{H}_{\nu }(x)=\sqrt{\pi }2^{\nu }x^{-\nu
-1}\Gamma \left( \nu +\frac{3}{2}\right) \mathbf{H}_{\nu }(x)$ belongs to
the Laguerre-P\'{o}lya class of real entire functions (since the exponential factors in the infinite product are canceled because of the symmetry of the zeros $\pm h_{\nu,n},$ $n\in\mathbb{N},$ with respect to the origin) and thus it satisfies
the Laguerre inequality (see \cite{skov})
\begin{equation}
\left( \mathcal{H}_{\nu}^{(n)}(x)\right) ^{2}-\left( \mathcal{H}%
_{\nu}^{(n-1)}(x)\right) \left( \mathcal{H}_{\nu}^{(n+1)}(x)\right) >0,
\label{H4}
\end{equation}%
where $\left\vert \nu \right\vert \leq \frac{1}{2}$ and $x\in \mathbb{R}$.
On the other hand, we have that%
\begin{equation*}
\mathcal{H}_{\nu}^{\prime }(z)=\sqrt{\pi }2^{\nu }z^{-\nu -2}\Gamma \left( \nu
+\frac{3}{2}\right) \left[ z\mathbf{H}_{\nu }^{\prime }(z)-(\nu+1)\mathbf{H}%
_{\nu }(z)\right] ,
\end{equation*}%
\begin{equation*}
\mathcal{H}_{\nu}^{\prime \prime }(z)=\sqrt{\pi }2^{\nu }z^{-\nu -3}\Gamma
\left( \nu +\frac{3}{2}\right) \left[ z^{2}\mathbf{H}_{\nu }^{\prime \prime
}(z)-2(\nu+1)z\mathbf{H}_{\nu }^{\prime }(z)+(\nu+1)(\nu+2)\mathbf{H}_{\nu }(z)%
\right] ,
\end{equation*}%
and thus the Laguerre inequality (\ref{H4}) for $n=1$ is equivalent to%
\begin{equation*}
\pi 2^{2\nu }x^{-2\nu -4}\left[ \Gamma \left( \nu +\frac{3}{2}\right) \right]
^{2}\left[ x^{2}\left( \mathbf{H}_{\nu }^{\prime }(x)\right) ^{2}-x^{2}%
\mathbf{H}_{\nu }(x)\mathbf{H}_{\nu }^{\prime \prime }(x)-\left(\nu+1\right)
\left( \mathbf{H}_{\nu }(x)\right)^{2}\right] >0.
\end{equation*}%
This implies that%
\begin{equation}
\left( \mathbf{H}_{\nu }^{\prime }(x)\right) ^{2}-\mathbf{H}_{\nu }(x)%
\mathbf{H}_{\nu }^{\prime \prime }(x)>\frac{\left(\nu+1\right) \left( \mathbf{%
H}_{\nu }(x)\right) ^{2}}{x^{2}}>0  \label{H5}
\end{equation}
for $\left\vert \nu \right\vert \leq \frac{1}{2}$ and $x\in \mathbb{R}$,
that is, the function $x\longmapsto \mathbf{H}_{\nu }^{\prime }(x)/\mathbf{H}%
_{\nu }(x)$ is decreasing on $(0,\infty )\backslash \left\{ h_{\nu,n}:n\in
\mathbb{N}\right\} .$ Since the zeros $h_{\nu,n}$ of the Struve function $%
\mathbf{H}_{\nu }$ are real and simple, $\mathbf{H}_{\nu }^{\prime }(x)$
does not vanish in $h_{\nu,n},$ $n\in \mathbb{N}$. Thus, for a fixed $k\in
\mathbb{N}$ the function $x\longmapsto \mathbf{H}_{\nu}^{\prime }(x)/%
\mathbf{H}_{\nu }(x)$ takes the limit $\infty $ when $x\searrow h_{\nu,k-1},$
and the limit $-\infty $ when $x\nearrow h_{\nu,k}.$ Moreover, since $%
x\longmapsto \mathbf{H}_{\nu}^{\prime}(x)/\mathbf{H}_{\nu}(x)$ is
decreasing on $(0,\infty )\backslash \left\{h_{\nu,n}:n\in \mathbb{N}\right\}
$ it results that in each interval $(h_{\nu,k-1},h_{\nu,k})$ its restriction
intersects the horizontal line only once, and the abscissa of this
intersection point is exactly $h_{\nu,k}^{\prime}$. So the zeros $h_{\nu,n}$
and $h_{\nu,n}^{\prime }$ interlace. Here we used the convention that $%
h_{\nu,0}=0$.

By means of (\ref{H1}) and (\ref{H3}) we have%
$$
\frac{z\mathbf{H}_{\nu }^{\prime }(z)}{\mathbf{H}_{\nu }(z)}%
=\nu+1-\sum_{n\geq1}\frac{2z^{2}}{h_{\nu,n}^{2}-z^{2}},\text{ \ \ \ }1+%
\frac{z\mathbf{H}_{\nu }^{\prime \prime }(z)}{\mathbf{H}_{\nu }^{\prime
}(z)}=\nu+1-\sum_{n\geq1}\frac{2z^{2}}{h_{\nu,n}^{\prime 2}-z^{2}}.
$$
and consequently
\begin{align*}
1+\frac{zu_{\nu}^{\prime \prime }(z)}{u_{\nu}^{\prime }(z)}&=1+\frac{z\mathbf{
H}_{\nu}^{\prime \prime }(z)}{\mathbf{H}_{\nu }^{\prime }(z)}+\left( \frac{1
}{\nu+1}-1\right) \frac{z\mathbf{H}_{\nu }^{\prime }(z)}{\mathbf{H}_{\nu }(z)
}\\&=1-\left( \frac{1}{\nu+1}-1\right) \sum_{n\geq1}\frac{2z^{2}}{h_{\nu,n}^{2}-z^{2}}%
-\sum_{n\geq1}\frac{2z^{2}}{h_{\nu,n}^{\prime 2}-z^{2}}.
\end{align*}
Now, suppose that $\nu\in \lbrack -\frac{1}{2},0].$ By using the inequality (%
\ref{s7}), for all $z\in \mathbb{D}_{h_{\nu,1}^{\prime }}$ we obtain the
inequality
$$
\Real\left( 1+\frac{zu_{\nu}''(z)}{u_{\nu}'(z)}%
\right) \geqslant 1-\left( \frac{1}{\nu+1}-1\right) \sum_{n\geq1}\frac{%
2r^{2}}{h_{\nu,n}^{2}-r^{2}}-\sum_{n\geq1}\frac{2r^{2}}{h_{\nu,n}^{\prime
2}-r^{2}},  \label{u2}
$$
where ${\left\vert z\right\vert =r.}$ Moreover, observe that if we use the
inequality (\ref{2.2}) then we get that the above inequality is also valid
when $\nu\in\left[0,\frac{1}{2}\right]$. Here we used that the zeros $h_{\nu,n}$ and $h_{\nu,n}^{\prime }$
interlace. The above inequality implies for ${r}\in \left( 0,h_{\nu,1}^{\prime
}\right) $%
\begin{equation*}
\inf_{z\in \mathbb{D}_{r}}\left\{ \Real\left( 1+\frac{zu_{\nu}^{\prime
\prime }(z)}{u_{\nu}^{\prime }(z)}\right) \right\} =1+\frac{ru_{\nu}^{\prime
\prime }(r)}{u_{\nu}^{\prime }(r)}.
\end{equation*}%
On the other hand, the function $U_{\nu}:\left( 0,h_{\nu,1}^{\prime }\right)
\longrightarrow \mathbb{R}$, defined by%
\begin{equation*}
U_{\nu}(r)=1+\frac{ru_{\nu}^{\prime \prime }(r)}{u_{\nu}^{\prime }(r)},
\end{equation*}%
is strictly decreasing since%
\begin{eqnarray*}
U_{\nu}^{\prime }(r) &=&-\left( \frac{1}{\nu+1}-1\right) \sum_{n\geq1}
\frac{4rh_{\nu,n}^{2}}{\left(h_{\nu,n}^{2}-r^{2}\right) ^{2}}%
-\sum_{n\geq1}\frac{4rh_{\nu,n}^{\prime 2}}{\left( h_{\nu,n}^{\prime
2}-r^{2}\right) ^{2}} \\
&<&\sum_{n\geq1}\frac{4rh_{\nu,n}^{2}}{\left( h_{\nu,n}^{2}-r^{2}\right)
^{2}}-\sum_{n\geq1}\frac{4rh_{\nu,n}^{\prime 2}}{\left( h_{\nu,n}^{\prime
2}-r^{2}\right) ^{2}}<0
\end{eqnarray*}%
for $\nu\in\left[0,\frac{1}{2}\right]$ and ${r}\in \left(
0,h_{\nu,1}^{\prime }\right),$ and also we have $U_{\nu}'(r)<0$ for $\nu\in\left[-\frac{1}{2},0\right]$ and $r>0.$
Here we used again that the zeros $h_{\nu,n}$
and $h_{\nu,n}^{\prime }$ interlace for all $n\in \mathbb{N}$, $\left\vert \nu
\right\vert \leq \frac{1}{2}$ and $r<\sqrt{h_{\nu,n}h_{\nu,n}^{\prime }}$ we
have that
\begin{equation*}
h_{\nu,n}^{2}\left( h_{\nu,n}^{\prime 2}-r^{2}\right) ^{2}<h_{\nu,n}^{\prime
2}\left( h_{\nu,n}^{2}-r^{2}\right) ^{2}.
\end{equation*}%
Since $\lim_{r\searrow 0}U_{\nu}(r)=1>\alpha $ and $\lim_{r\nearrow
h_{\nu,1}^{\prime }}U_{\nu}(r)=-\infty ,$ in view of the minimum principle for harmonic functions it follows that for $z\in \mathbb{D}%
_{r_{4}}$ we have%
\begin{equation*}
\Real\left( 1+\frac{zu_{\nu}^{\prime \prime }(z)}{u_{\nu}^{\prime }(z)}%
\right) >\alpha
\end{equation*}%
if and only if $r_{4}$ is the unique root of
\begin{equation*}
1+\frac{ru_{\nu}^{\prime \prime }(r)}{u_{\nu}^{\prime }(r)}=\alpha
\end{equation*}%
situated in $\left( 0,h_{\nu,1}^{\prime }\right) .$

\textbf{b)} By using
$$1+\frac{zv_{\nu }^{\prime \prime }(z)}{v_{\nu }^{\prime }(z)}=
-\nu+\frac{z\mathbf{H}_{\nu }^{\prime \prime }(z)}{\mathbf{H}_{\nu }^{\prime }(z)}=-\sum_{n\geq1}\frac{2z^{2}}{h_{\nu,n}^{\prime 2}-z^{2}}$$ and the
inequality (\ref{s7}), it follows that for all $z\in \mathbb{D}_{h_{\nu,1}^{\prime }}$ we
have
\begin{equation*}
\Real\left( 1+\frac{zv_{\nu }^{\prime \prime }(z)}{v_{\nu }^{\prime }(z)}
\right) \geqslant -\sum_{n\geq1}\frac{2r^{2}}{h_{\nu,n}^{\prime 2}-r^{2}},
\end{equation*}
where ${\left\vert z\right\vert =r.}$ Thus, for ${r}\in \left(
0,h_{\nu,1}^{\prime }\right) $ we get%
\begin{equation*}
\inf_{z\in \mathbb{D}_{r}}\left\{ \Real\left( 1+\frac{zv_{\nu }^{\prime
\prime }(z)}{v_{\nu }^{\prime }(z)}\right) \right\} =1+\frac{rv_{\nu
}^{\prime \prime }(r)}{v_{\nu }^{\prime }(r)}.
\end{equation*}%
On the other hand, the function $V_{\nu}:\left(0,h_{\nu,1}^{\prime }\right) \longrightarrow
\mathbb{R}$, defined by%
\begin{equation*}
V_{\nu}(r)=1+\frac{rv_{\nu }^{\prime \prime }(r)}{v_{\nu }^{\prime }(r)},
\end{equation*}%
is strictly decreasing \ and $\lim_{r\searrow 0}V_{\nu}(r)=1>\alpha$ and $%
\lim_{r\nearrow h_{\nu,1}^{\prime }}V_{\nu}(r)=-\infty $. Consequently, in view of the minimum principle for harmonic functions for $z\in
\mathbb{D}_{r_{5}}$ we have that
\begin{equation*}
\Real\left( 1+\frac{zv_{\nu }^{\prime \prime }(z)}{v_{\nu }^{\prime }(z)}%
\right) >\alpha
\end{equation*}%
if and only if $r_{5}$ is the unique root of
\begin{equation*}
1+\frac{rv_{\nu }^{\prime \prime }(r)}{v_{\nu }^{\prime }(r)}=\alpha
\end{equation*}%
situated in $\left( 0,h_{\nu,1}^{\prime }\right) .$

\textbf{c)} Similarly, by using
\begin{equation*}
1+\frac{zw_{\nu }^{\prime \prime }(z)}{w_{\nu }^{\prime }(z)}=-\frac{\nu}{2}+
\frac{\sqrt{z}\mathbf{H}_{\nu }^{\prime \prime }(\sqrt{z})}{2\mathbf{H}%
_{\nu }^{\prime }(\sqrt{z})}=-\sum_{n\geq1}\frac{2z}{h_{\nu,n}^{\prime 2}-z}.
\end{equation*}
and the inequality (\ref{s7}), we get for all $z\in \mathbb{D}_{h_{\nu,1}^{\prime }}$
\begin{equation*}
\Real\left( 1+\frac{zw_{\nu }^{\prime \prime }(z)}{w_{\nu }^{\prime }(z)}%
\right) \geqslant -\sum_{n\geq1}\frac{2r}{h_{\nu,n}^{\prime 2}-r},
\end{equation*}
where ${\left\vert z\right\vert =r.}$ Hence, for ${r}\in \left(
0,h_{\nu,1}^{\prime }\right)$ we obtain
\begin{equation*}
\inf_{z\in \mathbb{D}_{r}}\left\{ \Real\left( 1+\frac{zw_{\nu }^{\prime
\prime }(z)}{w_{\nu }^{\prime }(z)}\right) \right\} =1+\frac{rw_{\nu
}^{\prime \prime }(r)}{w_{\nu }^{\prime }(r)}.
\end{equation*}%
On the other hand, the function $W_{\nu}:\left( 0,h_{\nu,1}^{\prime }\right) \longrightarrow
\mathbb{R}$, defined by%
\begin{equation*}
W_{\nu}(r)=1+\frac{rw_{\nu }^{\prime \prime }(r)}{w_{\nu }^{\prime }(r)},
\end{equation*}%
is strictly decreasing and $\lim_{r\searrow 0}W_{\nu}(r)=1>\alpha$ and $%
\lim_{r\nearrow h_{\nu,1}^{\prime }}W_{\nu}(r)=-\infty .$ Consequently, in view of the minimum principle for harmonic functions for $z\in
\mathbb{D}_{r_{6}}$ we have that
\begin{equation*}
\Real\left( 1+\frac{zw_{\nu }^{\prime \prime }(z)}{w_{\nu }^{\prime }(z)}%
\right) >\alpha
\end{equation*}%
if and only if $r_{6}$ is the unique root of
\begin{equation*}
1+\frac{rw_{\nu }^{\prime \prime }(r)}{w_{\nu }^{\prime }(r)}=\alpha
\end{equation*}%
situated in $\left( 0,h_{\nu,1}^{\prime }\right) .$
\end{proof}

\section{\bf Some results on the zeros of the derivatives of Lommel and Struve functions}
\setcounter{equation}{0}

In this section our aim is to discuss the results on the derivatives of the Lommel and Struve functions of the first kind which we used in the proof of the main results. In the proof of Theorem \ref{theo1} we used the fact that the zeros of the function $s_{\mu-\frac{1}{2},\frac{1}{2}}$ and its derivative interlace when $\mu\in(0,1).$ In the same proof we also used the fact that the zeros of the function $s_{\mu-\frac{3}{2},\frac{1}{2}}$ and its derivative interlace when $\mu\in(0,1).$ In other words, this means that {\em the zeros of the Lommel function $s_{\mu-\frac{1}{2},\frac{1}{2}}$ and its derivative interlace when $\mu\in(-1,1),$ $\mu\neq0.$} We note that since the zeros $\xi_{\mu,n}$ of the Lommel function are real and simple when $\mu\in(-1,1),$ $\mu\neq0,$ the above result actually implies that {\em the zeros $\xi_{\mu,n}'$ of the function $s_{\mu-\frac{1}{2},\frac{1}{2}}'$ are also real when $\mu\in(-1,1),$ $\mu\neq0.$} This result is a natural companion to the results stated in \cite{Bar2,Kou} about the zeros of Lommel functions. Moreover, it is worth to mention that by using a famous result of P\'olya we can say more. P\'olya's result \cite{pol} is as follows:

{\em Suppose that the function $f$ is positive, increasing and continuous on $[0, 1)$ and that
$\int_{0}^{1}f(t)dt<\infty$. Then, the entire functions
$$z\longmapsto\int_{0}^{1}f(t)\sin(zt)dt\ \ \ \mbox{and} \ \ \ z\longmapsto\int_{0}^{1}f(t)\cos(zt)dt$$
have only real and simple zeros and their zeros interlace. }

In view of this result of P\'olya we can prove actually that {\em the zeros $\xi_{\mu,n}'$ of the function $s_{\mu-\frac{1}{2},\frac{1}{2}}'$ are all real and simple when $\mu\in(-1,1),$ $\mu\neq0.$} To show this we use the last expression of Lemma \ref{lem1} for $k=0$ and the derivative formula \eqref{phider} to obtain
$$\mu(\mu+1)z^{\frac{1}{2}-\mu}s_{\mu-\frac{1}{2},\frac{1}{2}}'(z)=z\varphi_0'(z)-\left(\mu+\frac{1}{2}\right)\varphi_0(z).$$
Combining this with \cite[Lemma 3]{Bar2}
$$z\varphi_0(z)=\mu(\mu+1)\int_0^1(1-t)^{\mu-1}\sin(zt)dt, \ \ \mu>0,$$
it results that
\begin{align*}
-z^{\frac{3}{2}-\mu}s_{\mu-\frac{1}{2},\frac{1}{2}}'(z)&=-\int_0^1(1-t)^{\mu-1}tz\cos(zt)dt+\left(\mu+\frac{3}{2}\right)\int_0^1(1-t)^{\mu-1}\sin(zt)dt\\
&=\int_0^1\left(\left(\mu+\frac{5}{2}\right)-\left(2\mu+\frac{3}{2}\right)t\right)(1-t)^{\mu-2}\sin(zt)dt,
\end{align*}
where in the last step we used integration by parts. Since for $t\in[0,1)$ and $\mu<1$ we have
$$k_{\mu}(t)=\left(\left(\mu+\frac{5}{2}\right)-\left(2\mu+\frac{3}{2}\right)t\right)(1-t)^{\mu-2}>0,$$
$$k_{\mu}'(t)=(1-\mu)\left(\left(\mu+\frac{7}{2}\right)-\left(2\mu+\frac{3}{2}\right)t\right)(1-t)^{\mu-3}>0,$$
we obtain that the function $k_{\mu}$ is positive and strictly increasing on $[0,1).$ Moreover, it can be shown that $\int_0^1k_{\mu}(t)dt<\infty.$ Summarizing, in view of P\'olya's result we have that {\em the zeros $\xi_{\mu,n}'$ of the function $s_{\mu-\frac{1}{2},\frac{1}{2}}'$ are all real and simple when $\mu\in(0,1).$}

Similarly as before, by using the last expression of Lemma \ref{lem1} for $k=1$ and the derivative formula 
$$\varphi_1'(z)=\mu(\mu-1)z^{-\mu-\frac{1}{2}}\left(zs_{\mu-\frac{3}{2}}'(z)-\left(\mu-\frac{1}{2}\right)s_{\mu-\frac{3}{2},\frac{1}{2}}(z)\right)$$
it follows that
$$\mu(\mu-1)z^{\frac{1}{2}-\mu}s_{\mu-\frac{3}{2},\frac{1}{2}}'(z)=z\varphi_1'(z)+\left(\mu-\frac{1}{2}\right)\varphi_1(z).$$
Combining this with \cite[Lemma 3]{Bar2}
$$\varphi_1(z)=\mu\int_0^1(1-t)^{\mu-1}\cos(zt)dt, \ \ \mu>0,$$
it results that
\begin{align*}
-(\mu-1)z^{\frac{1}{2}-\mu}s_{\mu-\frac{3}{2},\frac{1}{2}}'(z)&=\mu\int_0^1(1-t)^{\mu-1}tz\sin(zt)dt-\left(\mu-\frac{1}{2}\right)\int_0^1(1-t)^{\mu-1}\cos(zt)dt\\
&=\mu\int_0^1\left(\frac{3}{2}-\mu-\frac{1}{2}t\right)(1-t)^{\mu-2}\cos(zt)dt,
\end{align*}
where in the last step we used integration by parts. Since for $t\in[0,1)$ and $\mu<1$ we have
$$l_{\mu}(t)=\left(\frac{3}{2}-\mu-\frac{1}{2}t\right)(1-t)^{\mu-2}>0,$$
$$l_{\mu}'(t)=-\left(\frac{1}{2}(1-\mu)t-\mu^2+\frac{7}{2}\mu-\frac{5}{2}\right)(1-t)^{\mu-3}>0,$$
we obtain that the function $l_{\mu}$ is positive and strictly increasing on $[0,1).$ Moreover, it can be shown that $\int_0^1l_{\mu}(t)dt<\infty.$ Thus, we proved that {\em the zeros $\zeta_{\mu,n}'$ of the function $s_{\mu-\frac{3}{2},\frac{1}{2}}'$ are all real and simple when $\mu\in(0,1),$} that is, {\em the zeros $\xi_{\mu,n}'$ of the function $s_{\mu-\frac{1}{2},\frac{1}{2}}'$ are all real and simple when $\mu\in(-1,0).$}

It is also worth to mention that in the proof of Theorem \ref{theo2} we used the fact that {\em the zeros of the function $\mathbf{H}_{\nu}$ and its derivative interlace when $|\nu|\leq\frac{1}{2}.$} We note that since {\em the zeros $h_{\nu,n}$ of the Struve function are real and simple when $|\nu|\leq\frac{1}{2},$} the above result actually implies that {\em the zeros $h_{\nu,n}'$ of the function $\mathbf{H}_{\nu}'$ are also real when $|\nu|\leq\frac{1}{2}.$} This result is a natural companion to the results stated in \cite{steinig} about the zeros of Struve functions. Moreover, it is worth to mention also that by using P\'olya's result we can say more: {\em the zeros $h_{\nu,n}'$ of the function $\mathbf{H}_{\nu}'$ are all real and simple when $|\nu|\leq\frac{1}{2}.$}
To see this we proceed as above. If we use the recurrence relation
$$x\mathbf{H}_{\nu-1}(x)=\nu\mathbf{H}_{\nu}(x)+x\mathbf{H}_{\nu}'(x)$$
and the integral representation
$$\mathbf{H}_{\nu}(x)=\frac{2\left(\frac{x}{2}\right)^{\nu}}{\sqrt{\pi}\Gamma\left(\nu+\frac{1}{2}\right)}\int_0^1(1-t^2)^{\nu-\frac{1}{2}}\sin(xt)dt$$
we obtain that
$$-\sqrt{\pi}2^{\nu-1}\Gamma\left(\nu+\frac{1}{2}\right)x^{1-\nu}\mathbf{H}_{\nu}'(x)=\int_0^1(1-\nu-\nu t^2)(1-t^2)^{\nu-\frac{3}{2}}\sin(xt)dt.$$
Now, observe that the function $q_{\nu}:[0,1)\to (0,\infty),$ defined by
$$q_{\nu}(t)=(1-\nu-\nu t^2)(1-t^2)^{\nu-\frac{3}{2}}$$
satisfies $$q_{\nu}'(t)=-2t(1-t^2)^{\nu-\frac{5}{2}}\left(\left(\nu-\frac{3}{2}\right)(1-\nu-\nu t^2)+\nu(1-t^2)\right)\geq0$$
for all $t\in[0,1)$ and $|\nu|\leq \frac{1}{2}$ since the function $r_{\nu}:[0,1)\to\mathbb{R},$ defined by $$r_{\nu}(t)=\left(\nu-\frac{3}{2}\right)(1-\nu-\nu t^2)+\nu(1-t^2),$$ satisfies
$$r_{\nu}'(t)>0\ \ \ \Longrightarrow\ \ \ r_{\nu}(t)<r_{\nu}(1)=\left(\nu-\frac{3}{2}\right)(1-2\nu)\leq0,\ \ \ \mbox{when}\ \ \nu\in\left[0,\frac{1}{2}\right],$$
$$r_{\nu}'(t)<0\ \ \ \Longrightarrow\ \ \ r_{\nu}(t)<r_{\nu}(0)=\left(\nu-\frac{3}{2}\right)(1-\nu)+\nu\leq0,\ \ \ \mbox{when}\ \ \nu\in\left[-\frac{1}{2},0\right].$$
Thus, the function $q_{\nu}$ is positive and increasing on $[0,1).$ Moreover, it can be shown that $\int_0^1q_{\nu}(t)dt<\infty.$ These together imply that indeed {\em the zeros $h_{\nu,n}'$ of the function $\mathbf{H}_{\nu}'$ are all real and simple when $|\nu|\leq\frac{1}{2}.$}

\subsection*{Acknowledgements} The research of \'A. Baricz was supported by the Romanian National Authority for Scientific Research, CNCS-UEFISCDI, under Grant
PN-II-RU-TE-2012-3-0190. The work of both authors was initiated in September 2014 when they visited the Mathematics Department of Kafkas University in Kars, Turkey. Both authors are grateful to Murat \c{C}a\u{g}lar, Erhan Deniz and Nizami Mustafa for their kind hospitality.

\end{document}